\documentclass{article}
\usepackage{amsfonts}
\usepackage{amsmath,amssymb}
\usepackage[hyphens]{url} 
\usepackage{hyperref}
\usepackage{authblk}

\newtheorem{theorem}{Theorem}

\newtheorem{case}{Case}
\newtheorem{definition}[theorem]{Definition}

\newtheorem{proposition}[theorem]{Proposition}

\newenvironment{proof}[1][Proof]{\noindent\textbf{#1.} }{\ \rule{0.5em}{0.5em}}
\begin{document}

\title{Constants of de Bruijn-Newman type in analytic number theory and
statistical physics}
\author{Charles M. Newman \thanks{Courant Institute, New York University and NYU Shanghai} and Wei Wu \thanks{University of Warwick}}

\maketitle

\begin{abstract}
One formulation in 1859 of the Riemann Hypothesis (RH) was that the Fourier transform
$H_f(z)$ of $f$ for $ z \in \mathbb{C}$ has only real zeros when $f(t)$ is a
specific function $\Phi (t)$. P\'{o}lya's 1920s approach to RH extended $H_f$ to
$H_{f,\lambda}$, the Fourier transform of $e^{\lambda t^2} f(t)$. We review 
developments of this approach to RH and related ones in statistical physics where
$f(t)$ is replaced by a measure $d \rho (t)$. P\'{o}lya's work together with 
1950 and 1976 results of de Bruijn and Newman, respectively, imply the existence
of a finite constant $\Lambda_{DN} = \Lambda_{DN} (\Phi)$ in $(-\infty, 1/2]$
such that $H_{\Phi,\lambda}$ has only real zeros if and only if 
$\lambda \geq  \Lambda_{DN}$; RH is then equivalent to $\Lambda_{DN} \leq 0$.
Recent developments include the Rodgers and Tao proof of the 1976 conjecture
that $\Lambda_{DN} \geq 0$ (that RH, if true, is only barely so) and the
Polymath~15 project improving the $1/2$ upper bound to about~$0.22$.  
We also present examples of $\rho$'s with differing $H_{\rho,\lambda}$ and
$\Lambda_{DN} (\rho)$ behaviors; some of these are new and based on a recent weak
convergence theorem of the authors.

\end{abstract}

\section{Introduction}
\label{sec:intro}

For $\rho$ a positive Borel measure on the real line with $\rho(-\infty,+\infty)<\infty$
and $\lambda$ real, define

\begin{equation}
H_{\rho,\lambda }\left( z\right) \, := \, \int_{-\infty }^{\infty }
%H_{\rho,\lambda} (z) \, := \,
e^{izt} e^{\lambda
t^{2}} \,d\rho(t).  
\label{Hrho}
\end{equation}%
When $\rho$ has a density $f(t)$ with respect to Lebesgue measure, this becomes

\begin{equation}
H_{f,\lambda }\left( z\right) \, := \, \int_{-\infty }^{\infty }
e^{izt} e^{\lambda
t^{2}} f(t)  \,dt .  
\label{Hf1}
\end{equation}%
Depending on the behavior of $d \rho$ (or $f$) as $t \to \pm \infty$, 
$H_{\rho,\lambda }\left( z\right)$ will be an entire function of complex $z$
either for all real $\lambda$ or for a semi-infinite interval
of $\lambda$ including at least $(-\infty,0)$. The focus of this paper is on the
subset ${\cal P}_\rho$ (or ${\cal P}_f$) of those $\lambda$ such that
$H_{\rho,\lambda }$ (or $H_{f,\lambda }$) is entire and all its 
zeros in the complex plane 
are purely real. 

The main interest in this question, from analytic number theory, is its
relation to the Riemann Hypothesis (RH). Taking $f$ to be a specific function 
$\Phi$ (see~(\ref{Phi}) below), RH is equivalent to all the zeros of 
$H_{\Phi,0 }$ being real --- i.e., equivalent to ${\cal P}_\Phi \ni 0$.
In Section~\ref{sec:history} we review, following the history from 1859 
\cite{Rie} to the present \cite{RT}, some of the
main developments of this approach to RH, including 
(see Subsections \ref{subsec:statmech}--\ref{subsec:newman}) those
with motivations arising from statistical physics.
Three key facts are 
\begin{enumerate}
\item $\lambda \in {\cal P}_\Phi$ and $\lambda' > \lambda$ imply 
$\lambda' \in {\cal P}_\Phi$; this is due to P\'{o}lya \cite{Pol2} --- see
Theorem~\ref{UF}.

\item $1/2 \in {\cal P}_\Phi$; this is due to De Bruijn \cite{DB} --- see
Theorem~\ref{Gausssuf}.

\item ${\cal P}_\Phi$ is bounded below; this is due to Newman \cite{New1} --- see
Theorem~\ref{finite}.
\end{enumerate}

It follows  from these three facts that there is a well-defined constant
in $(-\infty, 1/2]$, now known as the de Bruijn-Newman constant, which
(for the case $f = \Phi$) we denote simply as $\Lambda_{DN}$, so that
${\cal P}_\Phi = [\Lambda_{DN},\infty)$ and { \it RH is equivalent to 
$\Lambda_{DN} \leq 0$\/}. In \cite{New1}, the proof that $\Lambda_{DN} > -\infty$
(i.e., fact number~3 above) was accompanied by the conjecture, complementary
to RH, that $\Lambda_{DN} \geq 0$, stated there as a {\it
``quantitative dictum that the Riemann Hypothesis, if true, is only barely so.''}

In Subsection~\ref{subsec:newman} we survey the history, starting
with \cite{CNV}, of improvements on the \cite{New1} lower bound
$\Lambda_{DN} > -\infty$ culminating with the 2018 proof by Rodgers
and Tao \cite{RT} of the 1976 conjecture that  $\Lambda_{DN} \geq 0$.
Then in Subsection~\ref{subsec:KKL} we discuss the shorter history
of improvements to the \cite{DB} upper bound $\Lambda_{DN} \leq 1/2$, consisting
of the Ki-Kim-Lee \cite{KKL} bound $\Lambda_{DN} < 1/2$ (and related results
such as Theorem~\ref{nonreal} below) 
as well as current efforts (see \cite{Poly15}) by Tao and collaborators to %at the time of this
%writing are on track to
obtain a concrete upper bound in $(0, 1/2)$, expected at the time of this
writing to be about $0.22$. 

In Section~\ref{sec:general} of this paper, we go beyond the RH case of
$d \rho = \Phi(x) dx$ and discuss a number of results and examples 
about ${\cal P}_\rho$ or ${\cal P}_f$ for more general $\rho$ and~$f$. 
This is natural from the statistical or mathematical physics point of
view and may be of general mathematical interest even without
an immediate connection to analytic number theory. 
As in the case of $\Phi$, we will always assume that $\rho$ and $f$ are
non-negative and even; sometimes $\rho$ will be normalized
to be a probability measure. 

A key result (see Theorem~\ref{weak}) discussed in 
Subsection~\ref{subsec:weakconverge} is a somewhat surprising weak convergence theorem
from \cite{NW} for sequences of probability distributions whose Fourier transforms
have only purely real zeros. The surprising aspect of this theorem is that
the purely real zeros property somehow controls tail behavior of the sequence 
of distributions in a uniform way. We make use of this theorem several times
in Section~\ref{sec:general}, including in Subsection~\ref{subsec:classify}
where we classify and give examples of measures $\rho$ with varying tail
behaviors and varying ${\cal P}_\rho$'s. The tail behaviors are subdivided in the
three subsections of \ref{subsec:classify} according to whether 
\begin{equation}
{\cal T}_{\rho} := \{ b \in (-\infty,+\infty): \, \int e^{bx^2} d \rho < \infty\}
\label{eqn:tailset}
\end{equation}
is $(-\infty, \infty)$ or $(-\infty,b_0)$ or $(-\infty,b_0]$. 

All possibilities for the pair $({\cal T}_{\rho}, {\cal P}_\rho)$ are
presented in Subsection~\ref{subsec:classify} in nine classes, labelled
as Cases \ref{firstcase} through~\ref{lastcase}. 
Concrete examples of a $\rho$ for each class are provided,
with the notable exception of Case~\ref{nexttolastcase} 
where it is shown that one cannot
have ${\cal T}_{\rho} = (-\infty, b_0]$ and ${\cal P}_\rho$ strictly bigger
than $\{b_0\}$; this is an interesting consequence
of the weak convergence result, Theorem~\ref{weak}.
%ADDED A SENTENCE HERE
Another consequence of Theorem~\ref{weak} is the construction in
Case~\ref{case3} of a $\rho$ with ${\cal T}_{\rho} = (-\infty,\infty)$ but
with ${\cal P}_\rho = \emptyset$ --- see in particular Proposition~\ref{case3prop} 
there. 

We conclude this introduction by noting that there are other
connections between the Riemann zeta function and statistical physics
than those we discuss. For example, random matrix theory seems to describe
very accurately the distribution of the zeros (as stated by Montgomery's conjecture \cite{Mon}) and the
fine asymptotics of the moments; see for example \cite{BK} for a review. More
recently, Fyodorov and Keating \cite{FK} also suggested that the extreme values of
the zeta function in short intervals behave as a log-correlated Gaussian field. See \cite{ABB,
Naj, AOR} for recent progress on these questions.

\section{History of the de Bruijn-Newman constant} 
\label{sec:history}

\subsection{ The Riemann Hypothesis and zeros of Fourier transforms}

In his fundamental paper \cite{Rie}, Riemann extended into the complex plane
the  $\zeta $
function, defined originally by Euler as %for all $s\in \mathbb{C}$ by 
\[
\zeta \left( s\right) =\sum_{n\geq 1}n^{-s} 
\]%
for real $s>1$.  
Obviously, $\zeta $ is holomorphic in the domain $\left\{ \Re{s}>1\right\} $%
. The $\zeta $ function is closely related to the distribution of primes,
through Euler's product expansion, $\sum_{n\geq 1}n^{-s}=\prod_{p}\left(
1-p^{-s}\right) ^{-1}$, over primes $p \geq 2$. A main achievement 
of \cite{Rie} is that the
function $\zeta $ admits an analytic continuation to the whole complex
plane (except for a simple pole at $s=1$), 
and the location of the zeros of $\zeta $ are closely related to
the asymptotic behavior as $x \to \infty$ of $\pi \left( x\right) 
$, the number of primes less than or equal to $x$.

We now sketch some of the basic ideas of \cite{Rie}. Starting with the 
identity for the Gamma function $\Gamma \left( z\right) $, 
\[
\int_{0}^{\infty }\exp \left( -\pi n^{2}x\right) x^{\frac{s}{2}-1}dx=\frac{%
\Gamma \left( s/2\right) }{\pi ^{\frac{s}{2}}n^{s}}\text{, \ \ }\Re{s}>0%
\text{, }n\in \mathbb{N}\text{,} 
\]%
one obtains%
\[
\pi ^{-\frac{s}{2}}\Gamma \left( \frac{s}{2}\right) \zeta \left( s\right)
=\int_{0}^{\infty }\psi \left( x\right) x^{\frac{s}{2}-1}dx, 
\]%
where 
\[
\psi \left( x\right) =\sum_{n=1}^{\infty }\exp \left( -\pi n^{2}x\right) . 
\]%
Notice that $\theta(x) :=2\psi(x)+1$ is the well-known Jacobi-theta function, which satisfies the identity 
%\[
\begin{equation}
\theta \left( x\right) =x^{-\frac{1}{2}}\theta \left( x^{-1}\right) .
\label{psiidentity}
\end{equation}
%\]%
%ADDED THE DERIVATION OF THIS IDENTITY 
The identity follows from the Poisson summation formula, which gives
\begin{multline*}
\theta \left( x\right)  = \sum_{n=-\infty}^{\infty }\exp \left( -\pi n^{2}x\right)
= \sum_{k=-\infty}^{\infty } \int_{-\infty}^{\infty } e^{ -\pi t^{2}x}
e^{-2\pi i kt} \, dt \\
= \sum_{k=-\infty}^{\infty } \frac{ e^{ -\pi k^{2}/x}}{\sqrt{x}}
=x^{-\frac{1}{2}}\theta \left( x^{-1}\right).
\end{multline*}
Riemann used \eqref{psiidentity} to obtain the representation  
\begin{equation}
\pi ^{-\frac{s}{2}}\Gamma \left( \frac{s}{2}\right) \zeta \left( s\right) =%
\frac{1}{s\left( s-1\right) }+\int_{1}^{\infty }\psi \left( x\right) \left(
x^{\frac{s}{2}-1}+x^{\frac{1-s}{2}-1}\right) dx\text{, }\Re{s}>1.
\label{zetarep}
\end{equation}%
Note that the integral above is uniformly convergent on compact subsets of $%
\mathbb{C}$, and therefore the left hand side admits an analytic
continuation in $\mathbb{C\setminus }\left\{ 0,1\right\} $. The function $%
\Gamma \left( \frac{s}{2}\right) $ has a pole at $0$, and simple poles at $%
-2,-4,...$; thus the values $-2k$, $k\in \mathbb{N}$, are simple
zeros of $\zeta $, known as the trivial zeros. Moreover, $\zeta $
does not vanish at other points in the half planes $\Re{s}<0$ or $
\Re{s}>1$. All other possible zeros of $\zeta $ are in the strip $0\leq 
\Re{s}\leq 1$, called the critical strip.

Using the change of variable $s=\frac{1}{2}+iz$, Riemann also introduced the 
$\xi $ function as
\begin{equation}
\xi \left( z\right) =s\left( s-1\right) \Gamma \left( \frac{s}{2}\right) \pi
^{-\frac{s}{2}}\zeta \left( s\right) .  \label{xi}
\end{equation}%
It follows from~(\ref{psiidentity}) and the above discussion
that $\xi $ is an even entire function. Furthermore, all
possible zeros of $\xi $ lie in the strip $\left\vert \Im{z}\right\vert
\leq \frac{1}{2}$.

In \cite{Rie} Riemann conjectured that $\zeta $ has infinitely many zeros in the
critical strip, and that one can explicitly represent%
\[
\pi \left( x\right) = li\left( x\right) +\sum_{\rho \in \mathcal{Z}\text{, }%
\Im{\rho} >0}\left( li\left( x^{\rho }\right) +li\left( x^{1-\rho
}\right) \right) +\int_{x}^{\infty }\frac{dt}{\left( t^{2}-1\right) \log t}%
-\log 2\text{, }x\geq 2, 
\]%
where $li(x): = \int_1^x (\ln t)^{-1} \,dt$ (with the integral defined at $t=1$ in the principal value sense) %is the integral logarithm 
% DEF OF INTEGRAL LOGARITHM ADDED
and $\mathcal{Z}$ is the set of zeros
of $\zeta $ in the critical strip. This formula was later rigorously
established by von Mangoldt. Riemann further conjectured that the zeros of
the $\xi $ function are all purely imaginary (or equivalently, all
nontrivial zeros of $\zeta $ are one the line $\Re{s}=\frac{1}{2}$).
This is the celebrated Riemann Hypothesis.

Using (\ref{zetarep}), (\ref{xi}), and a change of variables, we have 
\[
\xi \left( z\right) =2\int_{0}^{\infty }\Phi \left( u\right) \cos zu\,du, 
\]%
where 
\begin{equation}
\Phi \left( u\right) =\sum_{n\geq 1}\left( 4\pi ^{2}n^{4}e^{9u/2}-6\pi
n^{2}e^{5u/2}\right) e^{-\pi n^{2}\exp \left( 2u\right) }.  \label{Phi}
\end{equation}%
It is not obvious that $\Phi \left( u\right) $ is even, but in fact it
is, as follows from~(\ref{psiidentity}). Then the $\xi $ function 
is just the Fourier transform  
of~$\Phi $:%
\[
\xi \left( z\right) =\int_{-\infty }^{\infty }e^{izu}\Phi \left( u\right)
\,du. 
\]
It is not hard to see that $\Phi(u) \geq 0$ for all~$u$; this positivity
is valid term by term on the RHS of~(\ref{Phi}) for $u\geq0$ and then
the even-ness of $\Phi$ implies positivity for $u \leq 0$.
Up to  a multiplicative constant $\Phi$ may be regarded as an even probability
density function.

Based on the considerations above, a natural approach to the Riemann
Hypothesis is to establish criteria for Fourier transforms of 
(sufficiently rapidly decreasing at $\pm \infty$) probability
densities to possess only real zeros, and to apply them to the Riemann $\xi $
function. As will be discussed later, this encouraged  
study of the distribution of zeros of Fourier transforms.

\subsection{P\'olya and universal factors}

Motivated by the Riemann Hypothesis, P\'olya systematically investigated the
question of when the Fourier transform of a function has only real roots (or
equivalently, the Laplace transform has only pure imaginary zeros).
The function~$\Phi $ related to the Riemann $\xi $ function seems
complicated, so P\'olya's starting point was to use
(relatively simple)
classes of entire functions whose zeros are all real.

In \cite{PS}, by applying a Theorem of Laguerre, P\'olya and Schur obtained
the following class of entire functions, whose zeros are all real:

\begin{theorem}
\bigskip The entire function $f\left( z\right) $ is a uniform limit of real
polynomials with only real zeros, if and only if
%I CHANGED FINITE TO GENERAL PRODUCT IN DISPLAYED EQUATION
\begin{equation}
f\left( z\right) =bz^{n}\exp \left( -\lambda z^{2}+\kappa z\right)
\prod_{k}\left( 1-\frac{z}{b_{k}}\right) \exp \left( \frac{z}{b_{k}}%
\right) ,  \label{LP}
\end{equation}%
where $b,\kappa \in \mathbb{R}$, $n\in \mathbb{N}$, $\lambda \geq 0$, $%
b_{k}\in \mathbb{R}$, the product can be finite or infinite
and $\sum_{k} b_{k}^{-2}<\infty $.
\end{theorem}

Functions of the form (\ref{LP}) are said to belong to the
Laguerre-P\'olya class and we write $f\in \mathcal{LP}$.
The Riemann Hypothesis is equivalent to the statement that $\xi \in 
\mathcal{LP}$. At the time, there was a hope that a complete
characterization of the functions in $\mathcal{LP}$ could lead to a proof
(or disproof) of the Riemann Hypothesis. However, the 
conditions for functions to belong to $\mathcal{LP}$ were too involved to be
applied to the Riemann $\xi $ function, and this approach was slowly abandoned.

Later in \cite{Pol2}, P\'olya investigated when entire functions of the form 
\begin{equation}
f\left( z\right) :=\int_{-\infty }^{\infty }F\left( t\right) e^{izt}\,dt
\label{FF}
\end{equation}%
have only real zeros. The complex valued function $F$ is assumed to be
locally integrable, and statisfy the conditions:

\begin{equation}
F\left( -t\right) ={\bar F} (t) \quad \forall t,  \label{PA1}
\end{equation}%
and, for all $t$, 
\begin{equation}
\left\vert F\left( t\right) \right\vert \leq A\exp \left( -\left\vert
t\right\vert ^{2+\alpha }\right) \text{, for some }A,\alpha >0\text{.}
\label{PA2}
\end{equation}

To prove his results, P\'olya introduced the notion of \textit{%
universal factors}. These are complex valued functions $\varphi \left(
t\right)$ for $t\in \mathbb{R}$, such that if %for any entire function 
(\ref{FF})~has
only real zeros, then the Fourier transform 
\[
\int_{-\infty }^{\infty }\varphi \left( t\right) F\left( t\right)
e^{izt}\,dt 
\]%
also has only real zeros. P\'olya obtained the following complete
characterization of universal factors.

\begin{theorem}
\label{UF}A real analytic function $\varphi \left( t\right) ,t\in \mathbb{R}$
is a universal factor if and only if its holomorphic extension $\varphi
\left( t\right) $ in $\mathbb{C}$ is such that $\varphi \left( iz\right) $
is an entire function of the form (\ref{LP}).
\end{theorem}

Taking $f\left( z\right) =e^{-\lambda z^{2}}$ in (\ref{LP}), a simple example
of a universal factor is $\varphi \left( t\right) =e^{Bt^{2}}$
with $B>0$.
In other words, if all roots of  
\begin{equation}
\int_{-\infty }^{\infty }e^{izt}F\left( t\right) \,dt  \label{e.Lap}
\end{equation}%
are real, then the same holds for $\int_{-\infty }^{\infty
}e^{izt}e^{Bt^{2}}F\left( t\right) \,dt$ for $B>0$. Another simple
example of a universal factor is $\varphi \left( t\right) =\cosh at$.

Using Theorem \ref{UF}, P\'olya obtained the following class of functions
that have only real zeros.

\begin{theorem}
\label{PRZ}Suppose that $F$ satisfies (\ref{PA1}) and (\ref{PA2}), and in
addition, has a holomorphic extension in a neighborhood of the origin. Then
the complex function 
\[
H\left( z\right) =\int_{0}^{\infty }F\left( t\right) t^{z-1}\,dt 
\]%
admits an analytic continuation as a meromorphic function in $\mathbb{C}$.
If $H$ does not have zeros in the region $\mathbb{C}\setminus (-\infty ,0]$
and $q\in \mathbb{N}$, then the entire function 
\[
\int_{-\infty }^{\infty }F\left( t^{2q}\right) e^{izt}\,dt 
\]%
has only real zeros.
\end{theorem}

By applying Theorems \ref{UF} and \ref{PRZ}, P\'olya concluded that the
entire functions 
\begin{eqnarray*}
&&\int_{-\infty }^{\infty }\cosh at\exp \left( -a\cosh t\right) e^{izt}\,dt,
\\
&&\int_{-\infty }^{\infty }\exp \left( -t^{2q}\right) e^{izt}\,dt ,
\end{eqnarray*}%
have only real zeros and so does the entire function 
%\[
\begin{equation}
\int_{-\infty }^{\infty }\exp \left( -at^{4q}+bt^{2q}+ct^{2}\right)
e^{izt}\,dt,\text{ }q\in \mathbb{N}\text{.} 
\label{Polyaexample}
\end{equation}
%\]

%I MADE SOME CHANGES HERE
In another paper of P\'olya \cite{Pol3}, he explained how the existence of
infinitely many real zeros is in general easier to prove than the
non-existence of complex zeros. In particular, a method of Hardy leads to
the following criterion:

\begin{proposition}
\label{inftyzero}Suppose that $F$ is an even function, analytic and real,
and such that $\lim_{t\rightarrow \infty }F^{\left( n\right) }\left(
t\right) t^{2}=0$ for $n=0,1,2,...$ If the function $f\left( z\right) $
defined by (\ref{FF}) has only a finite number of real zeros, then there is
an integer $N$ such that $\left\vert F^{\left( n\right) }\left( it\right)
\right\vert $ is an increasing function if $n>N$ and $0<t<T$, where $iT$ is
the singular point of $F\left( t\right) $ which is closest 
to the origin.
\end{proposition}

\subsection{De Bruijn's results}

In \cite{DB}, de Bruijn continued P\'olya's line of research and studied the
zeros of entire functions of the form (\ref{FF}), under the conditions (\ref%
{PA1}) and (\ref{PA2}). Examples of such entire functions include 
\[
\int_{-\infty }^{\infty }e^{-t^{2n}}e^{izt}\,dt,\text{ }n\in \mathbb{N} 
\]%
and 
\[
\int_{-\infty }^{\infty }\exp \left( -a\cosh t\right) e^{izt}\,dt,\text{ }%
a>0. 
\]%
Two of the main results of \cite{DB} are the next two theorems.

\begin{theorem}
\label{DB1}\bigskip Let $f\left( t\right) $ be an entire function such
that its derivative $f^{\prime }\left( t\right) $ is the limit (uniform in
any bounded domain) of a sequence of polynomials, all of whose roots lie on
the imaginary axis. Suppose further that $f$ is not a constant, $f\left(
t\right) =f\left( -t\right) $, and $f\left( t\right) \geq 0$ for $t\in 
\mathbb{R}$. Then the integral%
\[
\int_{-\infty }^{\infty }\exp \left( -f\left( t\right) \right) e^{izt}\,dt 
\]%
has real roots only.
\end{theorem}

The proof of this result relies on the following.

\begin{theorem}
\label{DB2}Let $N\in \mathbb{N}$, and let 
\[
P\left( t\right) =\sum_{n=-N}^{N}p_{n}e^{nt}\text{, }\Re p_{n}>0\text{, }%
p_{-n}=\bar{p}_{n}\text{, }n=1,...,N. 
\]%
Let the function $q\left( z\right) $ be regular in the sector $-\pi
/2N-N^{-1}\arg p_{N}<\arg z<\pi /2N-N^{-1}\arg p_{N}$ and on its boundary,
with the possible exception of $z=0$ and $z=\infty $ which may be poles (of
arbitrary finite order) for $q\left( z\right) $. Furthermore suppose 
$\bar{q}(z)
%q\left( z\right) }
=q\left( 1/\bar{z}\right) $ in this sector (in other
words, $q\left( z\right) $ is real for $\left\vert z\right\vert =1$). Then
all but a finite number of roots of the function%
\begin{equation}
\Psi \left( z\right) =\int_{-\infty }^{\infty }\exp \left( -P\left( t\right)
\right) q\left( e^{t}\right) e^{izt}\,dt  \label{psi}
\end{equation}%
are real.
\end{theorem}

A key ingredient used to prove Theorems \ref{DB1} and \ref{DB2} was the notion of 
a \textit{strong universal factor}. Suppose that $F$ satisfies (\ref{PA1}),
(\ref{PA2}) and the complex function $S\left( t\right) $ satisfies the
following two properties.

\begin{enumerate}
\item If the roots of (\ref{FF}) lie in a strip $\left\vert \Im%
{z}\right\vert \leq \Delta $, for some $\Delta >0$, then those of 
\begin{equation}
\int_{-\infty }^{\infty }F\left( t\right) \,S\left( t\right) e^{izt}\,dt
\label{suf}
\end{equation}
lie in a strip $\left\vert \Im{z}\right\vert \leq \Delta _{1}$, where
the constant $\Delta _{1}<\Delta $ is independent of $F$.

\item If, for any $\varepsilon >0$, all but a finite number of roots of (\ref%
{FF}) lie in the strip $\left\vert \Im{z}\right\vert \leq \varepsilon $,
then the function (\ref{suf}) has only a finite number of non-real roots.
\end{enumerate}
Then $S\left( t\right) $ is called a strong universal factor. By
definition, any strong universal factor is a universal factor in
P\'olya's sense.
%MADE SOME CHANGES HERE; ARE THEY OK?
The function $S\left( t\right) =\cosh t$ is a simple example of a strong
universal factor. As discussed earlier, the Gaussian density $e^{Bt^{2}}$ is
a universal factor and the following theorem of \cite{DB} 
shows it has Property~1. It was later proved in \cite{KKL} to also have Property~2.
%currently it is still not clear whether it has property 2.

\begin{theorem}
\label{Gausssuf}Suppose that the function $F$ satisfies (\ref{PA1}), (\ref%
{PA2}) and the zeros of the entire function (\ref{FF}) lie in the
strip $\left\vert \Im{z}\right\vert \leq \Delta $. Then all the roots of
the entire function 
\begin{equation}
\int_{-\infty }^{\infty }F\left( t\right) \,e^{\lambda t^{2}/2}e^{izt}\,dt
\label{Gsuf}
\end{equation}%
lie in the strip 
\begin{equation}
\left\vert \Im{z}\right\vert \leq \left[ \max \left( \Delta ^{2}-\lambda
,0\right) \right] ^{1/2}.  \label{strip}
\end{equation}
\end{theorem}

De Bruijn also showed that a large class of functions are strong universal
factors. In particular, functions of the type 
\begin{equation}
S\left( t\right) =\sum_{n=-N}^{N}a_{n}e^{n\lambda t}\text{, }a_{n}=\bar{a}%
_{-n}\text{, }\lambda >0  \label{sufpoly}
\end{equation}%
are strong universal factors if all their roots lie on the imaginary axis.
Conversely, if a function of the form (\ref{sufpoly}) is a strong universal
factor, then all of its roots lie on the imaginary axis. In particular, 
de Bruijn proved:

\begin{theorem}
\label{SUFpoly}Suppose the roots of the function (\ref{sufpoly}) lie on the
imaginary axis. Then if the roots (resp., all but a finite number of the roots) of (%
\ref{FF}) lie in the strip $\left\vert \Im{z}\right\vert \leq \Delta $,
then the roots (resp., all but a finite number of the roots) of~(\ref%
{suf}) lie in the strip $\left\vert \Im{z}\right\vert \leq \left[ \max
\left( \Delta ^{2}-\lambda ^{2}N/2,0\right) \right] ^{1/2}$.
\end{theorem}

%I MOVED THE START OF THE NEXT SUBSECTION TO HERE

\subsection{The de Bruijn-Newman constant}

Here is a direct consequence of Theorem \ref{Gausssuf}. Recall that the function $%
\Phi $ in (\ref{Phi}) is (proportional to) 
a probability density related to the Riemann $\xi $
function. Define for any $\lambda \in \mathbb{R}$, $z\in \mathbb{C}$, 
\begin{equation}
H_{\lambda }\left( z\right) =\int_{-\infty }^{\infty }e^{\lambda t^{2}}\Phi
\left( t\right) e^{izt}\,dt.  \label{Hzeta}
\end{equation}%
Applying %Theorem \ref{Gausssuf}, we conclude 
Polya's results (see the discussion following Theorem~\ref{UF} above),
it follows that there exists a constant $%
\Lambda _{DN}$ (but at this stage of our discussion, it
could potentially be $\pm \infty $), such that $H_{\lambda }$
has all real zeros\ if and only if $\lambda \geq \Lambda _{DN}$. Moreover,
since the roots of $\xi $ lie in the strip $\left\vert \Im%
{z}\right\vert \leq \frac{1}{2}$, de Bruijn's quantitative bound~(\ref{strip}) gives
an upper bound $\Lambda _{DN}\leq 1/2$.

More generally, one can define the functions $H_{f,\lambda }$ for a 
density $f$ and $H_{\rho,\lambda }$ for a measure $\rho$ by (\ref{Hf1})
and~(\ref{Hrho}) and then a corresponding $\Lambda _{DN} (f)$ or 
$\Lambda _{DN} (\rho)$. This will be the topic of Section~\ref{sec:general}
below.

We return to the constant $\Lambda _{DN}$ defined after Eq.~(\ref{Hzeta}).
%applying Theorem \ref{Gausssuf}, 
The result of de Bruijn that $\Lambda _{DN}\leq
1/2 $ did not exclude the possibility that $\Lambda
_{DN}=-\infty $. 
The first lower bound for $\Lambda _{DN}$ was given in \cite{New1}:
%The paper of Newman \cite{New1} gives a lower bound for $%
%\Lambda _{DN}$.

\begin{theorem}
\label{finite}There exists a real number $\Lambda _{DN}$ with $-\infty
<\Lambda _{DN}\leq 1/2$, such that $H_{\lambda }$ defined by (\ref{Hzeta})
has only real zeros when $\lambda \geq \Lambda _{DN}$ but has nonreal zeros
when $\lambda <\Lambda _{DN}$.
\end{theorem}

The constant $\Lambda _{DN}$ is now known as the de Bruijn-Newman constant.
By definition, the Riemann Hypothesis is equivalent to
having the zeros of $H_{0}$
all purely real. Therefore it is equivalent to $\Lambda _{DN}\leq 0$%
. In \cite{New1}, the complementary conjecture was made that $\Lambda
_{DN}\geq 0$, %which implies $\Lambda _{DN}=0$. This conjecture as a
as a quantitative version of the notion that 
``the Riemann Hypothesis, if true, is only barely
so\textquotedblright .

Theorem \ref{finite} was proved in \cite{New1} by giving a 
complete characterization of
all even probability measures $\rho $ such that for any $\lambda >0$ the
Fourier transform of $\exp \left( -\lambda t^{2}\right) d\rho \left(
t\right) $ has only real zeros, as stated in the next theorem. It may be of some
interest to note that the original motivation for this result was
provided more by statistical physics and Euclidean field theory, as
discussed in Section 2.5, than by analytic number theory.

\begin{theorem}
\label{-infty}Let $\rho$ be an even probability measure and let $H_{\rho,\lambda }$
be defined by~(\ref{Hrho}). Then $H_{\rho,\lambda }$ has only real zeros for all $%
\lambda \in \mathbb{R}$ if and only if either 
\[
d\rho\left( t\right) =\frac{1}{2}\left( \delta \left( t-t_{0}\right) +\delta
\left( t+t_{0}\right) \right) \text{, for some }t_{0}\geq 0\text{,} 
\]%
or  $d\rho = f(t) dt$ with 
%$f$ is absolutely continuous with respect to the density%
\begin{equation}
f(t) = Kt^{2m}\exp \left( -\alpha t^{4}-\beta t^{2}\right) \prod_{j}\left[ \left( 1+%
\frac{t^{2}}{a_{j}^{2}}\right) e^{-\frac{t^{2}}{a_{j}^{2}}}\right] .
\label{DBNdensity}
\end{equation}%
Here $K>0$, $m=0,1,...,a_{j}>0$, $\sum \frac{1}{a_{j}^{4}}<\infty $, $%
\alpha >0$ and $\beta \in \mathbb{R}$ (or $\alpha =0$ and $\beta +\sum \frac{%
1}{a_{j}^{2}}>0$).
\end{theorem}

Theorem \ref{finite} follows from Theorem \ref{-infty} by showing that $\Phi(t)$ is not of the
form \eqref{DBNdensity}. This was done in \cite{New1} by using the elementary inequality
\begin{equation*}
(1+x^2)e^{-x^2} \ge e^{-x^4/2}
\end{equation*}
for real $x$, to show that $f(t)$ given by \eqref{DBNdensity} satisfies
\begin{equation*}
f(t) \geq K t^{2m} \exp {(- [\alpha + \sum_j 1/(2a_j^4)]t^4
- \beta t^2)} ,
\end{equation*}
which decays much more slowly as $|t| \to \infty$ than does $\Phi(t)$. 

When $\alpha = 0$ in the above expression for $f(t)$, 
$H_{f,\lambda}$  is required to have only real zeros for
those $\lambda$ such that $H_{f,\lambda}$ is an entire function. E.g., if
there is only a finite product over $j$ in~(\ref{DBNdensity}), then those
are the $\lambda$ in $(-\infty, \beta + \sum (1/\alpha_j ^2) )$.

It is worth noting that the density (\ref{DBNdensity}) is reminiscent
of the Laguerre-P\'olya class (\ref{LP}). Also, (\ref{DBNdensity}) has a
natural physical interpretation --- the theorem implies that in
order to have $\Lambda _{DN}=-\infty $, $f$ can only be either a discrete
density with two point masses, or a density that is a kind of perturbation of the
$\phi ^{4}$ measure (i.e., the $q=1$ case of \eqref{Polyaexample}). See the next section for more details.

\subsection{Zeros of the partition function in statistical mechanics and
quantum field theory}
\label{subsec:statmech}

In this subsection we take a detour and discuss examples of probability
distributions arising in statistical mechanics and Euclidean (quantum) field theory,
whose Laplace transform has only pure imaginary zeros. In the 1950s,
Yang and Lee 
studied partition functions of a lattice gas with a variable
chemical potential, or equivalently, of an Ising model in a 
variable external magnetic field %in
%the thermodynamic limit 
\cite{LY1, LY2}. They discovered a striking property
for the locations of zeros of the Ising partition function --- all the roots 
in the complex plane for the magnetic field variable lie
on the imaginary axis. To be more precise, consider a collection of
${\pm 1}$--valued 
random variables $\left\{ X_{j}\right\} _{j=1}^{N}$ 
whose joint distribution is given by the Gibbs
measure 
\begin{equation}
%d\mu =\frac{1}{Z}\exp \left( \sum_{i,j=1}^{N}J_{ij}X_{i}X_{j}\right)
d\mu = K \,
\exp \left( \sum_{i,j=1}^{N}J_{ij}x_{i}x_{j}\right)
\prod_{j=1}^{N} %\delta \left( X_{j}=\pm 1\right) ,  \label{Ising}
(\delta (x_j -1) + \delta (x_j +1)) ,
\label{Ising}
\end{equation}%
where $J_{ij}\geq 0$ and $K$ is the normalizing constant that makes (\ref%
{Ising}) a probability measure; then 
\[
\int \exp \left( it\sum_{i=1}^{N}\lambda _{i}x_{i}\right) d\mu \text{, \ }%
\lambda _{i}\geq 0,
\]%
as a function of complex $t$ has only real zeros. The original proof, in \cite{LY2}, expressed this transform as a multinomial
in the variables $\{e^{\pm i t \lambda_j} \}$ and used fractional
linear transformation arguments. Quite different proofs may be found, e.g.,
in Sec. 3 of \cite{New2} and in  \cite{LS}.

%I'M NOT SURE WHETHER WE SHOULD KEEP THE NEXT THEOREM IN THE PAPER???
%IF WE DO KEEP IT, PROBABLY SHOULD REPLACE VARIABLE $\sigma$  
%BY A DIFFERENT LETTER
%LIKE  $s$. 
%I KEPT THE NEXT THEOREM BUT CHANGED NOTATION

The result of Lee and Yang has been extended to more general Ising-type models 
in which $\delta(x_j - 1) + \delta (x_j + 1)$ is replaced by some
other measure $d \mu_0 (x)$. For example, in the study of lattice
$\phi ^4$ Euclidean field theories, $d \mu_0 (\phi)$ would be of the form
$\exp (-a \phi^4 + b \phi ^2) d \phi$; a good general reference for this
topic is \cite{FFS}.
One of the extensions of the Lee-Yang result, established in \cite{New2} (see also \cite{LS}), is as follows. 
%Among them the best result is established by Newman \cite{New2}, and
%we state the theorem below.

\begin{theorem}
\label{genIsing}Suppose the Hamiltonian, $H\left( S \right)$ for
$S \in \mathbb{R}^N$ is 
\[
H\left( S \right) = - \sum_{i,j=1}^{N}J_{ij}S_{i}S_{j} 
\]%
with $J_{ij}=J_{ji}\geq 0$, $\forall i,j$, and $\mu _{0}$ is an arbitrary
signed measure on $\mathbb{R}$ that is even or odd, has the property that $%
\int_{-\infty }^{\infty }e^{bs^{2}}\,d\left\vert \mu _{0}\left( s
\right) \right\vert <\infty $ for all~$b$, and satisfies the condition 
\[
\int_{-\infty }^{\infty }e^{hs }\,d\mu _{0}\left( s \right) \neq 0%
\text{, \ \ if }\Re{h}>0 \, ;
\]%
then, for all $\beta \geq 0$, 
\[
\int e^{\beta \sum h_{i}s_{i}}e^{-\beta H\left( S \right) }\prod
\,_{i=1}^{N}d\mu _{0}\left( s_{i}\right) \neq 0\text{, \ \ if }\Re%
{h_{i}}>0\text{ }\forall i. 
\]%
It follows that 
\[
F\left( z\right) :=\int e^{iz\sum \lambda _{i}s_{i}}e^{-\beta H\left(
S \right) }\prod \,_{i=1}^{N}d\mu _{0}\left( s_{i}\right) \text{,
\ with} \, \lambda _{i}\geq 0 \,  \forall i
\]%
has only real zeros. If the integral property above is only valid for
small $b>0$, then the conclusions remain valid, but only for small $\beta>0$.
\end{theorem}

Lee-Yang type theorems also arise in the context of constructive quantum
field theory. Indeed, an important special case of Theorem \ref{genIsing} is 
when
\[
d\mu _{0}\left( s \right) =\exp \left( -as ^{4}-bs
^{2}\right) d s ,\text{ }a>0\text{ and }b\in \mathbb{R}\text{.} 
\]%
Theorem \ref{genIsing} when combined with P\'olya's result
about (\ref{Polyaexample}) with $q=1$ then implies the 
Lee-Yang property for continuum $%
\left( \phi ^{4}\right) $ fields. This was actually first established 
by Simon and
Griffiths \cite{SG} in a different way 
using an approximation by classical Ising models. In the
same paper, they also showed the Lee-Yang property fails for some measures of
the form 
\[
d\mu \left( s \right) =\exp \left( -as ^{6}-bs ^{4}-cs
^{2}\right) ds ,\text{ }a>0. 
\]

%I DELETED THE REST OF THIS SUBSECTION.

%We close this section with a brief discussion of
%multi-component spin models in statistical
%mechanics. Consider classical $N-$component spins, $N=2,3$. Given a finite
%graph $G=\left( \mathcal{V},\mathcal{E}\right) $, we denote the spin
%variable at each $v\in \mathcal{V}$ by $S_{v}=\left(
%S_{v}^{1},...,S_{v}^{N}\right) $ in the unit sphere. Given real numbers $%
%\left\{ J_{e}\right\} _{e\in \mathcal{E}}$, the classical $O\left( N\right) $
%model on $G$ is defined by a Gibbs measure 
%\begin{equation}
%Z_{G}^{-1}\exp \left( -\beta H\left( S\right) \right) \prod_{v\in \mathcal{V}%
%}\delta (\left\vert S_{v}\right\vert =1),  \label{cXY}
%\end{equation}%
%with the Hamiltonian $H\left( S\right) $ given by 
%\begin{equation}
%H\left( S\right) =-\sum_{e=\left( i,j\right) \in \mathcal{E}}J_{e}\text{ }%
%S_{i}\cdot S_{j},  \label{XYham}
%\end{equation}%
%where $S_{i}\cdot S_{j}=\sum_{\alpha =1}^{N}S_{i}^{\alpha }S_{j}^{\alpha }$,
%and $\ Z_{G}$ is the normalization constant that makes (\ref{cXY}) a
%probability measure.
%
%\begin{theorem}[\protect\cite{DN}; see also \cite{LS} for a better proof when $N=2$]            %\cite{DN, LS}]
%Let $N=2,3$. Suppose that $J_{e}\geq 0$ for all $e\in \mathcal{E}$ and $%
%\lambda _{v}\geq 0$ for all $v\in \mathcal{V}$, then 
%\[
%\mathbb{E}\left[ \exp \left( iz\sum_{v\in \mathcal{V}}\lambda
%_{v}S_{v}^{1}\right) \right] 
%\]%
%has only real zeros.
%\end{theorem}

\subsection{Lower bounds for $\Lambda _{DN}$ and the proof of the
\cite{New1} conjecture that $\Lambda _{DN} \geq 0$}
\label{subsec:newman}

%Attempts on proving Newman's conjcture $\Lambda _{DN}\geq 0$ 
Improvements to the \cite{New1} result that $\Lambda _{DN} > -\infty$ have been made
since the late 1980s, giving a series of quantitative lower bounds for $\Lambda
_{DN}$ that supported the conjecture that $\Lambda _{DN}\geq 0$. 
%SOME CHANGES IN THE WORDING OF REST OF PARAGRAPH
The first result in this direction was by
Csordas, Norfolk and Varga \cite{CNV}, who proved $\Lambda _{DN}\geq -50$.
This was followed by \cite{Riel, NRV, CRV} and then a paper 
by Csordas, Smith and Varga 
\cite{CSV}, which showed $\Lambda _{DN}\geq
-4.379\times 10^{-6}$. There were further improvements
in \cite{COSV, Od}, and then by Saouter, Gourdon, and Demichel \ 
\cite{SGD} who showed $\Lambda _{DN}\geq -1.15\times 10^{-11}$.
%I THINK WE SHOULD CITE ALL THE PAPERS AS DONE BY RODGERS AND TAO.

We briefly summarize the methodology behind these proofs, in particular the
approach of \cite{CSV}. Assume for now $\Lambda _{DN}<0$ and so that also the Riemann
Hypothesis holds. The lower bound for $\Lambda _{DN}$ in \cite{CSV} was
established by exploiting the following repulsion phenomenon: if $\Lambda
_{DN}$ were significantly less than zero, then adjacent zeros of $H_{0}$ (or $%
\xi $) could not be too close to each other, which would contradict 
known facts about close pairs of zeros of $\xi $. More precisely, let $%
0<x_{1}<x_{2}<...$ be the non-negative zeros of $\xi $, and let $x_{-j}=-x_{j}$.
\ Define%
\[
g_{k}:=\sum_{j\neq k,k+1}\left[ \frac{1}{\left( x_{k}-x_{j}\right) ^{2}}+%
\frac{1}{\left( x_{k+1}-x_{j}\right) ^{2}}\right] . 
\]%
The following lower bound is established in \cite{CSV}.

\begin{theorem}
The de Brujin-Newman constant satisfies%
\[
\forall k , \,
\Lambda _{DN}\geq \lambda _{k}\text{, } %\forall k 
\]%
where 
\[
\lambda _{k}:=\frac{\left( 1-\frac{5}{4}\left( x_{k+1}-x_{k}\right)
^{2}g_{k}\right) ^{4/5}-1}{8g_{k}}. 
\]
\end{theorem}

Then lower bounds for $\Lambda _{DN}$ can be obtained by studying 
``Lehmer pairs" of zeros \cite{Le}, which are, roughly speaking,
two consecutive simple zeros on
the critical line that are exceptionally close to each other. The results of 
\cite{CSV} and \cite{SGD} were obtained by first numerically locating Lehmer
pairs of increasingly high quality. In principle, by locating infinitely
many Lehmer pairs with arbitrary close gaps, the method of \cite{CSV} could
verify the conjecture $\Lambda _{DN}\geq 0$. However, known
upper bounds on gaps of zeros of $\xi $ were not sufficient to make that
strategy work.

More recently, Rodgers and Tao \cite{RT} proved that $\Lambda _{DN}\geq 0$,
thus confirming the Newman's conjecture. Instead of looking at individual
pairs of zeros, they focused on zeros in intervals of the type $\left[
T,T+\alpha \right] $, where $1 \ll \alpha \ll \log T$ and proved that they exhibit a
kind of relaxation to local equilibrium --- if $\Lambda _{DN}<0$, then
the zeros of $H_{0}$ would be arranged locally as an approximate arithmetic
progression. To illustrate, observe that the function 
\[
H_{\lambda }\left( z\right) =\int_{-\infty }^{\infty }e^{\lambda t^{2}}\Phi
\left( t\right) e^{izt}\,dt 
\]%
satisfies the backward heat equation, 
\begin{equation}
\partial _{t}H_{t}=-\partial _{zz}^{2}H_{t},  \label{BHE}
\end{equation}%
with terminal condition $H_{0}$. This was noted, perhaps for the first time, in \cite{CSV},
where it was also observed that for $t > \Lambda_{DN}$,
the real zeros $x_k(t)$ obey the system of ODEs,
\begin{equation}
\label{zerodynamic}
\partial _{t}x_k=2 \sum_{j\neq k} \frac1{x_k - x_j}.
\end{equation}
A one-parameter family of solutions
to (\ref{BHE}), that are %stationary after dividing by some $h(t)$ 
called the equilibrium state,
is given by $H_{t}\left( z\right) =e^{tu^{2}}\cos zu$, for $%
u>0$, whose zeros are all arranged as an arithmetic progression $\left\{ \frac{%
2\pi \left( k+\frac{1}{2}\right) }{u}:k\in \mathbb{Z}\right\} $. As discussed in Section 4 of \cite{RT}, \eqref{zerodynamic} is reminiscent
of Dyson Brownian motion and of similar behavior that was studied by Erdos, Schlein and Yau, in the context of
gradient flow for the eigenvalues of random matrices \cite{ESY}.

%I THINK IN THE FOLLOWING PARAGRAPHS AND THEOREM WE MAY BE SAYING TOO
%MUCH ABOUT THE RODGERS-TAO ARGUMENTS. CAN WE WRITE A MUCH SHORTER SKETCH,
%LIKE ONE PARAGRAPH AND NO THEOREM?
%IS THE FOLLOWING OK?
%I CHANGED SOME OF THE WORDING AND DELETED PART IN THE MIDDLE OF THE REST OF THIS SUBSECTION

We now sketch some of the main steps in \cite{RT}. Assuming that $\Lambda _{DN}<0$
(and hence also assuming the validity of RH),
their goal was to show a contradiction by obtaining 
some control on the distribution of the zeros of $H_{t}$
when $\Lambda _{DN}/2\leq t\leq 0$. To do this, they exploit an observation in \cite{CSV} that the
dynamics for $H_{t}$ induces a gradient flow for the zeros $\left\{ x_{j}\left(
t\right) \right\} $. Indeed, for $\Lambda _{DN}<t\leq 0$, one can define a Hamiltonian by 
\[
\mathcal{H}\left( t\right) \mathcal{=}\sum_{j,k:j\neq k}\log \frac{1}{%
\left\vert x_{j}\left( t\right) -x_{k}\left( t\right) \right\vert }
\]%
and then view the evolution of $H_{t}$ as the gradient flow of 
$\mathcal{H}$,
\begin{equation}
\partial _{t}\mathcal{H}\left( t\right) \mathcal{=}-4E\left( t\right) \, ,
\label{gradflow}
\end{equation}%
where $E\left( t\right) $ is defined, roughly speaking, as 
\begin{equation}
E\left( t\right) =\sum_{j\neq k}\frac{1}{\left\vert x_{j}\left( t\right)
-x_{k}\left( t\right) \right\vert ^{2}}.  \label{energy}
\end{equation}%
Since~(\ref{energy}) is only a formal sum,
Rodgers and Tao studied a mollified version 
and used that together with~(\ref{gradflow}) to show
that $\{x_j (t)\}$ at $t=0$ would satisfy
%\[
%\tilde{E}_{jk}\left( t\right) :=\frac{1}{\left\vert x_{j}\left( t\right)
%-x_{k}\left( t\right) \right\vert ^{2}}-\frac{1}{\left\vert \xi _{j}-\xi
%_{k}\right\vert ^{2}}+2\frac{\left( x_{k}\left( t\right) -\xi _{k}\right)
%-\left( x_{j}\left( t\right) -\xi _{j}\right) }{\left( \xi _{j}-\xi
%_{k}\right) ^{3}}, 
%\]%
%where $\xi _{j}$ 
%is
%defined by 
%\[
%\frac{\xi _{j}}{4\pi }\log \frac{\xi _{j}}{4\pi }-\frac{\xi _{j}}{4\pi }=j. 
%\]%
%They also define for any discrete interval $I\subset \mathbb{Z}$, 
%\[
%\tilde{E}^{I}\left( t\right) =\sum_{j,k\in I,j\neq k}\tilde{E}_{jk}\left(
%t\right) . 
%\]%
%By exploiting (\ref{gradflow}), they were able to control  the
%mollified energy at time~$0$, by 
%\begin{equation}
%\tilde{E}^{\left[ T\log T,2T\log T\right] }\left( 0\right) = o\left(
%T\log _{+}^{3}T\right) .  \label{energy0}
%\end{equation}%
%Using the definition of $\tilde{E}$, (\ref{energy0}) implies that the
%zero locations satisfy%
%%\[
%%\frac{x_{j+1}\left( 0\right) -x_{j}\left( 0\right) }{\xi _{j+1}-\xi _{j}}%
%=1+o_{T}\left( 1\right) \text{, if }j\in \left[ T\log T,2T\log T\right] , 
%%\]%
%%or%
\[
x_{j+1}\left( 0\right) -x_{j}\left( 0\right) =\frac{4\pi +o_{T}\left(
1\right) }{\log T}\text{, if }j\in \left[ T\log T,2T\log T\right], 
\]%
for a fraction $1 -o(1)$ of these zeros --- see Sec. 9 of \cite{RT}. This would imply (under the assumption $\Lambda_{DN} < 0$) that the spacings 
between the zeros of the zeta function are
rarely much larger or much smaller than in an arithmetic progression. However
this would contradict a consequence of a result of Montgomery \cite{Mon}, who analyzed the
pair correlations for the zeros assuming the Riemann Hypothesis. The results of \cite{Mon} imply 
that a positive proportion of the spacings  
are strictly smaller than the mean spacing; see also \cite{CGG}. That 
completes our very rough outline of their proof.

\subsection{The upper bound for $\Lambda _{DN}$}
\label{subsec:KKL}

Recall that the Riemann Hypothesis is equivalent to the inequality $\Lambda
_{DN}\leq 0$: Based on the properties of universal multipliers, de Bruijn
proved the upper bound $\Lambda _{DN}\leq 1/2$. This has been improved more
recently by Ki, Kim and Lee to $\Lambda _{DN}<1/2$ \cite{KKL} along
with other results, as we
now discuss.

The starting observation of \cite{KKL} is that $S\left( t\right) =e^{\lambda
t^{2}}$ is indeed a strong universal multiplier in de Bruijn's sense,
namely, if for some $\varepsilon >0$, all but a finite number of roots of (%
\ref{FF}) lie in the strip $\left\vert \Im{z}\right\vert \leq
\varepsilon $, then for any $\lambda > 0$, 
the function (\ref{Gsuf}) has only a finite number of
non-real roots. A further observation by \cite{KKL} is that for certain 
functions $F$ (roughly speaking, those where the number of non-real zeros 
of~(\ref{FF}) is less
than the number of real zeros), the equality in (\ref{strip}) cannot be
reached. More precisely:

\begin{theorem}
\label{strictsuf}Suppose that $f\left( z\right) =\int_{-\infty }^{\infty
}F\left( t\right) e^{izt}\,dt$ is a real entire function of order less than~2,
$f$ has finitely many non-real zeros, and the number of non-real zeros of 
$f$ in the upper half plane does not exceed the number of real zeros. % of $f$.
Suppose also that $\Delta _{0}>0$ and the zeros of $f$ lie in the strip $%
\left\{ z:\left\vert \Im{z}\right\vert \leq \Delta _{0}\right\} $. If $%
\lambda \in \left( 0,\Delta _{0}^{2}/2\right) $, then the zeros of 
\[
\int_{-\infty }^{\infty }F\left( t\right) e^{\lambda t^{2}}e^{izt}\,dt 
\]%
lie in $\left\{ z:\left\vert \Im{z}\right\vert \leq \Delta _{1}\right\} $
for\ some $\Delta _{1}<\sqrt{\Delta _{0}^{2}-2\lambda }$.
\end{theorem}

Then, using saddle point methods and the properties of strong universal
multipliers, \cite{KKL} proved:

\begin{theorem}
\label{nonreal}For any $\lambda >0$, all but finitely many zeros of $%
H_{\lambda }\left( \cdot \right) $ are real and simple.
\end{theorem}

Note that $H_{\lambda }$ has infinitely many real zeros by 
results of P\'olya \cite{Pol3}
(see Proposition~\ref{inftyzero} above and the preceding
discussion). Combining this with Theorem \ref%
{strictsuf}, one concludes the upper bound $\Lambda _{DN}<1/2$. The proof
presented in \cite{KKL} suggested but did not give a quantitative improvement for $%
\Lambda _{DN}$ beyond $\Lambda _{DN} < 1/2$.

Improved upper bounds for $\Lambda _{DN}$ have been discussed in detail
by Tao and collaborators
in the Polymath~15 project \cite{Poly15}. 
%SHORTENED THE DESCRIPTION BELOW, IS IT OK?
%I SHORTENED IT SOME MORE
They combine the methods of \cite{DB}  and \cite{KKL} with
extensive numerical computations; 
%They used an idea going back to de Bruijn, 
%$that the non-real zeros of $H_{t}$ are attracted to the real line as $t$
%increases.
%%that the non-real zeros of $H_{t}$ are attracted to the real line as $t$
%increases. More precisely, without assuming the Riemann Hypothesis, the
%zeros $\left\{ z_{j}\left( t\right) \right\} $ of $H_{t}$ satisfy 
%\[
%\frac{d}{dt}z_{j}\left( t\right) =-2\sum_{j\neq k}\frac{1}{z_{j}\left(
%t\right) -z_{k}\left( t\right) }. 
%\]%
%If we write $z_{j}\left( t\right) =x_{j}\left( t\right) +iy_{j}\left(
%t\right) $, then 
%\[
%\frac{d}{dt}y_{j}=2\sum_{j\neq k}\frac{y_{k}-y_{j}}{\left(
%x_{k}-x_{j}\right) ^{2}+\left( y_{k}-y_{j}\right) ^{2}}. 
%\]%
%that the imaginary parts of zeros satisfy the equation
%\begin{equation}
%\frac{d}{dt}y_{j}=-\frac{1}{y_{j}}+2\sum_{k\neq j,j^{\prime }}\frac{%
%y_{k}-y_{j}}{\left( x_{k}-x_{j}\right) ^{2}+\left( y_{k}-y_{j}\right) ^{2}}.
%\label{Im}
%end{equation}
%The strategy then is to combine both analytical 
%and numerical work to exploit the zero free regions of $H_t$, thus improving the
%estimate for the attraction time. 
at the time of the preparation of this survey paper, the
upper bound was 
$\Lambda _{DN}\leq 0.22.$

\section{De Bruijn-Newman constant for general measures}
\label{sec:general}

In this section we study the function $H_{\mu,\lambda }(z)$ 
defined as in~(\ref{Hrho}) 
%
%\[
%H_{\mu,\lambda }\left( z\right) =\int_{-\infty }^{\infty }e^{\lambda
%t^{2}} e^{izt}\,d\mu(t) . 
%\]%
for a general even probability measure $\mu$. We will classify
$\mu$'s according to the zeros of $H_{\mu,\lambda}(z) $ and
the de Bruijn-Newman constant $\Lambda _{DN}\left( \mu \right)$, 
defined as follows.
As in Section~\ref{sec:intro}, we define
\begin{equation}
{\cal P}_{\mu} := \{\lambda: H_{\mu,\lambda}(z) \text{ is entire
and has only real zeros }\}. 
\label{DBNset}
\end{equation}
If ${\cal P}_{\mu}$  is nonempty, then $\Lambda _{DN}\left( \mu \right) $ 
is defined as its infimum.
If ${\cal P}_{\mu}$  is empty but
$H_{\mu,\lambda}(z)$ is entire for all $\lambda$,
we define $\Lambda _{DN}\left( \mu \right) $ to be $+ \infty$;
in the remaining case, $\Lambda _{DN}\left( \mu \right) $ is undefined.
%then $\Lambda _{DN}\left( \mu \right) $ is defined as its infimum. 
%When ${\cal P}_{\mu}$ is empty, we will 
%say that $\Lambda _{DN}\left( \mu \right) $ is undefined except
%in the situation where $H_{\mu,\lambda}(z)$ is entire for all $\lambda$,
%where we will define $\Lambda _{DN}\left( \mu \right) $ to be $+ \infty$.
When $d \mu (t) = f(t) dt$, we write $\Lambda _{DN}\left( f \right) $ 
and ${\cal P}_f$.  %${\cal L}_{DN}\left( \mu \right) $
Two main
ingredients we use are Theorem \ref{weak} below, proved in \cite{NW}, 
based on the only real zeros property being preserved under 
convergence; and the properties of strong universal factors
(e.g., Theorem \ref{Gausssuf}) studied by de Bruijn and others.

\subsection{A weak convergence theorem}
\label{subsec:weakconverge}

Let $\mu $ be a probability measure on $\mathbb{R}$ and $X$ be a random
variable on some probability space $\left( \Omega ,\mathcal{F},\mathbb{P}%
\right) $ with distribution $\mu $.
Motivated by the statistical physics results discussed in 
Subsection~\ref{subsec:statmech}, we define a collection ${\cal X}$
of probability measures as follows.
We use here the standard probability theory notation
with $\mathbb{E}[h(X)]$ denoting $\int_{-\infty}^{+\infty} h(t) d \mu (t).$

\begin{definition}
\label{LYp}We say $\mu $ (or $X$) is in ${\mathcal X}$ %of Lee-Yang type 
if it has the following three properties:
\end{definition}

\begin{enumerate}
\item $X$ has the same distribution as $-X$,

\item $\mathbb{E}\left[ \exp \left( bX^{2}\right) \right] <\infty $ for some 
$b>0$,

\item the function  
$\mathbb{E}\left[ \exp \left( izX\right) \right]$, defined for all
$z\in \mathbb{C}$, has
 only real zeros.
\end{enumerate}

The next theorem, from \cite{NW}, 
states that the combination of these three properties
(and in particular,
the only real zeros property of the Fourier transform) is preserved under weak
convergence (i.e., pointwise convergence
of the Fourier transform on the real axis). 
It helps explain why the sub-Gaussian Property~2 is built
into Definition~\ref{LYp}.

\begin{theorem}
\label{weak}Suppose for each $n\in \mathbb{N}$, $\mu _{n}\in \mathcal{X}$
and $\mu _{n}$ converges weakly to the probability measure $\mu $. Then $\mu 
\in \mathcal{X}$.
\end{theorem}

This theorem seems surprising at first glance for the following reason.
Since $\mu_n \in \mathcal{X}$ for each $n$, there is some $b_n > 0$
so that $X_n$ distributed by $\mu_n$ satisfies 
$\mathbb{E}\left[ \exp \left( b_n X^{2}\right) \right] <\infty $.
But without assuming that $b_n$ is bounded away from zero as $n \to \infty$,
why should it be that the limit $X$ has 
$\mathbb{E}\left[ \exp \left( b X^{2}\right) \right] <\infty $ for
some $b>0$? The answer, in brief, is that the purely real zeros Property~3
somehow implies that $b_n$ can be bounded away from zero or else
weak convergence would fail. 

%SHOULD WE ADD A SKETCH OF THE PROOF OR IS THE ABOVE PARAGRAPH ENOUGH?
%ADDED NEXT PARAGRAPH

The key to the proof of Theorem~\ref{weak} is the product representation,
for $X \in \mathcal{X}$,
\begin{equation}
\mathbb{E}\left[ \exp \left( izX\right) \right] \, = \,
e ^{-B z^2} \prod_k (1 - z^2/y_k^2) \, ,
\label{productrep}
\end{equation}
with $B \geq 0$, $y_k \in (0, \infty)$ and 
\begin{equation}
\mathbb{E}\left[ X^2 \right] \, = \,
2 (B+ \sum_k (1/y_k^2)) \, < \, \infty \, .
\label{secondmoment}
\end{equation}
One then shows that weak convergence for $X_n \in \mathcal{X}$ 
(distributed by $\mu_n$) requires a uniform bound first for
$\mathbb{E}\left[ X^2 \right] $ and then that this yields
a uniform bound away from zero for the sequence $\{b_n\}$ discussed above. 

%In particular, the limiting measure must satisfy $\mu \left( r,\infty
%\right) \leq \exp \left( -cr^{2}\right) $ for some $c>0$. 
The next theorem,
proved in \cite{NW}, relates the all real zeros property to the distribution tail
behavior and explains further why the sub-Gaussian Property~2 is natural
in Definition \ref{LYp}. The theorem follows directly from a theorem of
Goldberg and Ostrovskii \cite{GO}.

\begin{theorem}
\label{slowtail2} Suppose the random variable $X$ satistifes the following
two properties: 
%IS $X$ EQUAL IN DISTRIBUTION TO $-X$ HERE?

\begin{enumerate}
\item $\mathbb{E} [e^{b\left\vert X\right\vert ^{a}} ] <\infty $ for some $b>0$
and $a>1,$

\item $\mathbb{E} [e^{b^{\prime }X^{2}} ] =\infty $ for all $b^{\prime }>0$.
\end{enumerate}

\noindent Then $\mathbb{E} [e^{izX}] $ has some zeroes that are not purely
real.
\end{theorem}

One can derive from Theorem~\ref{weak} a different
result than Theorem~\ref{slowtail2} which also shows 
that the Fourier transform of certain
distributions do have some non-real zeros.

\begin{proposition}
\label{noPIZ}Let $\rho $ be an even probability measure such that $%
\int_{-\infty }^{\infty }e^{bt^{2}}\,d\rho \left( t\right) =\infty $, for
any $b>0$. Then for any $\lambda <0$,%
\[
G_{\lambda }\left( z\right) :=\int_{-\infty }^{\infty }e^{izt}e^{\lambda
t^{2}}d\rho \left( t\right) 
\]%
has some zeros that are not real. Thus ${\cal P}_\rho$ is either $\{0\}$
and $\Lambda_{DN} (\rho) = 0$ or else ${\cal P}_\rho$ is empty.
\end{proposition}

\begin{proof}
Fix  $\lambda <0$ and suppose that $G_{\lambda }\left( z\right) $ has only
real zeros. Take a sequence $\left\{ \lambda _{n}\right\} $ such that $%
\lambda _{0}=\lambda $, $\left\{ \lambda _{n}\right\} $ is increasing and $%
\lambda _{n}\rightarrow 0$ as $n\rightarrow \infty $. Applying Theorem \ref%
{Gausssuf} we conclude that $G_{\lambda _{n}}\left( z\right) $ has only real
zeros for all $n$. The measure $e^{\lambda _{n}t^{2}}d\rho \left( t\right)
/G_{\lambda _{n}}\left( 0\right) $ is clearly even and satisfies %
Property~2 of Definition~\ref{LYp} with $b=-\lambda _{n}/2$, 
and therefore $e^{\lambda
_{n}t^{2}}d\rho \left( t\right) /G_{\lambda _{n}}\left( 0\right) \in 
\mathcal{X}$. Since $e^{\lambda _{n}t^{2}}d\rho \left( t\right) /G_{\lambda
_{n}}\left( 0\right) $ converges weakly to $\rho $, we can apply Theorem \ref{weak}
to conclude that $\rho \in \mathcal{X}$. But this contradicts the fact
that $\int_{-\infty }^{\infty }e^{bt^{2}}\,d\rho \left( t\right) =\infty $
for any $b>0$.
\end{proof}

As a consequence of Proposition \ref{noPIZ}, we may construct distributions
whose Fourier transforms have nonreal zeros.
Examples are distributions with density $g\left( x\right) e^{-\lambda
x^{2}}$, with $\lambda >0$, $g\geq 0$ even and $\int e^{bx^{2}}g\left(
x\right) \,dx=\infty $ for any $b>0$. Specific cases include 
\[
%\text{Const.}
K\,
e^{-a\left\vert x\right\vert -\lambda x^{2}}\text{ \ with }a>0,\lambda >0, 
\]%
and 
\[
%\text{Const.} 
K\,
\left( 1+x^{2}\right) ^{-\theta} e^{-\lambda x^{2}}\text{ \ with }\theta >%
\frac{1}{2}\text{ and }\lambda >0. 
\]

\subsection{Classifying probability measures by $\Lambda_{DN}(\rho)$
and ${\cal P}_\rho$}
%according to $H_{f,\protect%
%\lambda }$}
\label{subsec:classify}

Using Theorem \ref{weak} on weak convergence, one may classify even 
distributions $\rho$  according to tail behavior and ${\cal P}_\rho$.
In this section we give examples of the various possibilities; they
are organized in three subsections according to tail 
behavior --- see~(\ref{eqn:tailset}) for the definition of ${\cal T}_\rho$.

%\subsubsection{$\protect\rho $ such that $\protect\int e^{bx^{2}}\,d\protect%
%\rho <\infty $, for all $b>0$.}
\subsubsection{${\cal T}_\rho = (-\infty, \infty)$ .}

We further discuss three cases, according to ${\cal P}_\rho$ --- see~(\ref{DBNset}).

\begin{case}
\label{firstcase}
%$\left\{ b:\int e^{izx}e^{bx^{2}}\,d\rho \text{ has only real zeros}\right\}
${\cal P}_\rho
=(-\infty ,\infty )$.
\end{case}

This class of probability distributions is completely characterized by
Theorem \ref{-infty}: either
\[
\rho \left( t\right) =\frac{1}{2}\left( \delta \left( t-t_{0}\right) +\delta
\left( t+t_{0}\right) \right) \text{, for some }t_{0}\geq 0\text{,} 
\]%
or $\rho$ is absolutely continuous with density%
\[
Kt^{2m}\exp \left( -\alpha t^{4}-\beta t^{2}\right) \prod_{j}\left[ \left( 1+%
\frac{t^{2}}{a_{j}^{2}}\right) e^{-\frac{t^{2}}{a_{j}^{2}}}\right] . 
\]
%Where $K>0$, $m=0,1,...,a_{j}>0$, $\sum \frac{1}{a_{j}^{4}}<\infty $, $%
%\alpha >0$ and $\beta \in \mathbb{R}$.
We note (see the discussion after 
Theorem~\ref{-infty}) that for $\alpha=0$, $\rho$ can belong to Case~\ref{middlecase} below.
%IS THAT THE RIGHT CASE?

\begin{case}
\label{case2}There exists $\Lambda _{0} \in (-\infty, \infty)$, 
such that 
%$\left\{ b:\int
%e^{izx}e^{bx^{2}}\,d\rho \text{ has only real zeros}\right\} 
${\cal P}_\rho=[\Lambda
_{0},\infty )$ $.$
\end{case}

The simplest example in this class is a $\left\{ \pm 1,0\right\} $ valued
symmetric random variable, for instance
\[
\rho \left( t\right) =\frac{1}{6}\left( \delta \left( t-1\right) +\delta
\left( t+1\right) \right) +\frac{2}{3}\delta \left( t\right) . 
\]%
One can explicitly calculate $\int e^{izx}e^{bx^{2}}\,d\rho =\frac{2}{3}+%
\frac{1}{6}e^{b}y+\frac{1}{6}e^{b}\frac{1}{y}$, where $y=e^{it}$. When $%
e^{b}<2$, it has roots where $y$ is real (and $\left\vert y\right\vert \neq 1$%
); whereas for $e^{b}\geq 2$, it has roots only where $\left\vert y\right\vert =1$
(thus $t$ is real). Therefore 
\begin{equation*}
\left\{ b:\int e^{izx}e^{bx^{2}}\,d\rho \text{
has only real zeros}\right\} =[\ln 2,\infty ).
\end{equation*}

Examples where $d\rho(x) = Kf(x)dx$ follow from the results discussed in Section~\ref{sec:history}.
These include $f(x)=\exp(-a\,\cosh (x))$ with $a>0$, $f(x) = \exp (-x^{2q})$ with
$q \in \{3, 4, 5, \dots\}$ and the RH case of $f(x) = \Phi(x) $.

\begin{case}
\label{case3}
%$\left\{ b:\int e^{izx}e^{bx^{2}}\,d\rho \text{ has only real zeros}\right\}
${\cal P}_\rho=\emptyset $.
\end{case}

We begin with a general proposition that is easily seen to follow from Theorem \ref{weak}
and P\'olya's result that $e^{bt^2}$ is a universal factor.

\begin{proposition}
\label{case3prop}
Suppose $\rho$ is even with $\cal{T}_\rho = (-\infty, \infty)$ 
such that for some $0<b_n \to \infty$ and $r_n >0$, 
\begin{equation*}
%\mbox{normalization} \times  
K_n  e^{b_nt^2} d\rho(t)|_{t=r_nu} \to d\mu(u),
\end{equation*}
with $\mu \notin \cal{X}$, where $K_n$ normalizes the
lefthand side to be a probability measure; then ${\cal P}_\rho =\emptyset$.
\end{proposition}

We next sketch how to construct such a $\rho$ of the form 
\begin{equation*}
d\rho(t) = %\mbox{normalization } \times 
K\, \sum_{k=1}^{\infty} a_k (\delta(t-d_k)+ \delta(t+d_k)),
\end{equation*}
with $d_k \nearrow \infty$ and $a_k \searrow 0$ rapidly, $r_n = d_{n+1}$ and 
\begin{equation}
\label{e.mu}
d\mu = (1/2) \delta (u) + (1/4) \delta (u-1)+ (1/4) \delta (u+1).
\end{equation}
That $\mu \notin \cal{X} $ follows from the explicit calculation of Case \ref{case2} above.

We start with $a_1=1$ and inductively construct first $a_{n+1}$ then $b_n$ for $n\ge1$. 
A key idea is that $b_n$ will be a solution $b$ of 

\begin{equation}
\label{e.b}
e^{b \, d^2_{n+1}} a_{n+1} = \sum_{k=1}^n e^{b \, d^2_{k}}a_k
\end{equation}
so that $e^{b_nt^2} d\rho$ will give equal mass 
to $\{d_{n+1}\}$ and $(0, d_n]$. Because we will also 
require $d_{k+1} / d_k \to \infty$ (say, $d_k = e^{k^2}$), this will lead to the 
limit $\mu$ of \eqref{e.mu}. The inductive choice of $a_{n+1}$ will imply that 
there is a unique smallest postive solution $b=B(a_{n+1}; a_1, \cdots, a_n)$ of \eqref{e.b}.

We now construct $a_n$ inductively. Choose $a_{n+1}$ as any $a>0$ (say, the largest) 
satisfying the following three inequalities (with the first two only for $n\ge 2$):
\begin{equation}
\label{induction1}
e^{b_{n-1}d^2_{n+1}} a_{n+1} \leq (1/(n+1)) e^{b_{n-1}d^2_{n}} a_{n},
\end{equation}
\begin{equation}
\label{induction2}
B(a_{n+1}; a_1, \cdots, a_n) \ge b_{n-1},
\end{equation}
\begin{equation}
\label{induction3}
B(a_{n+1}; a_1, \cdots, a_n) \ge n+ 1 \, .
\end{equation}
Then, as previously indicated, we choose $b_n=B(a_{n+1}; a_1, \cdots, a_n)$. This completes our sketch 
%I ADDED SOMETHING HERE
except to note that the $n+1$ appearing on the right hand sides of (\ref{induction1}) and (\ref{induction3})
could be replaced by any sequence $\theta_n$ with $\theta_n >1$ and $\theta_n \nearrow \infty$.

\subsubsection
{${\cal T}_\rho = (-\infty, b_0)$ with $b_0 \in (0, \infty)$.}

\begin{case}
\label{middlecase}
%$\left\{ b:\int e^{izx}e^{bx^{2}}\,d\rho \text{ has only real zeros}\right\}
${\cal P}_\rho =(-\infty,b_{0})$.
\end{case}

Here one can take the Gaussian measure $d\rho \left( s\right) /ds \, =\left(
%b_{0}^{2}/2\pi \right) ^{-1/2}
b_{0}/\pi \right)^{1/2}
e^{-b_{0}s^{2}}$.
%IS THIS CORRECT; ALSO PLEASE CHECK THE NORMALIZATION ABOVE. 
%I CHANGED A SIGN

\begin{case}
There exists $\Lambda _{0} \in (-\infty, b_0)$ such that 
%$\left\{ b:\int
%e^{izx}e^{bx^{2}}\,d\rho \text{ has only real zeros}\right\} 
${\cal P}_\rho =[\Lambda 
_{0},b_{0})$.
\end{case}

Take $b_{0} \in (0, \infty) $ large and let $d\rho ( t) :=\frac{1%
}{5} ( 3\delta ( t) +\delta ( t-1) +\delta (
t+1) ) dt$. Also take $d\nu _{b_{0}}( t) =(
b_{0}/\pi ) ^{-1/2}e^{-b_{0}t^{2}}dt$ and let $\rho =\mu \ast \nu
_{b_{0}}\text{.}$  Then
\[
d\rho \left( x\right) =%\text{normalization} \times  
K \, e^{-b_{0}x^{2}}\left[
3+e^{-b_{0}}e^{2b_{0}x}+e^{-b_{0}}e^{-2b_{0}x}\right] dx
\]%
is such that $\int e^{bx^{2}}\,d\rho <\infty $ if and only if $b<b_{0}$.
By an explicit computation, 
\begin{multline*}
\int e^{izx}e^{bx^{2}}\,d\rho
=%\text{Const.} 
K \,
e^{-\frac{z^{2}}{2}\frac{1}{2\left( b_{0}-b\right) }} \times \\
\left[ 3+e^{-\left(
b_{0}-\frac{b_{0}^{2}}{b_{0}-b}\right) }e^{i\frac{b_{0}}{b_{0}-b}%
z} +e^{-\left( b_{0}-\frac{b_{0}^{2}}{b_{0}-b}\right) }e^{ -i%
\frac{b_{0}}{b_{0}-b}z} \right] .
\end{multline*}%
Let $y=e^{i\frac{b_{0}}{b_{0}-b}z} $. It is not hard to check
that the roots of the equation $%
e^{-\left( b_{0}-\frac{b_{0}^{2}}{b_{0}-b}\right) }\left( y+y^{-1}\right)
+3=0$ all satisfy $\left\vert y\right\vert =1$ if and only
if $ e^{-\left( b_{0}-\frac{b_{0}^{2}}{b_{0}-b}\right)
} \geq 3/2$, i.e., $b\geq \Lambda _{0}$ for some $\Lambda
_{0}=\Lambda _{0}\left( b_{0}\right) >0$. Therefore 
\begin{equation*}
\left\{ b:\int
e^{izx}e^{bx^{2}}\,d\rho \text{ has only real zeros}\right\} =[\Lambda
_{0},b_{0}).
\end{equation*}

\begin{case}
${\cal P}_\rho
%\left\{ b:\int e^{izx}e^{bx^{2}}\,d\rho \text{ has only real zeros}\right\}
=\emptyset  \, .$
\end{case}

Let 
\[
d\rho \left( x\right) = %\text{normalization}\times 
K \, e^{-x^{2}}\left(
1+x\right) dx.
\]
Then $e^{bx^{2}}\,d\rho $ is integrable  on
$(-\infty, \infty)$ if and only if $b<1$ so $b_0 =1$. By an explicit
computation, for any $b<1$ and letting $\alpha =1-b$, we have 
\begin{eqnarray*}
\int_{-\infty }^{\infty }e^{izx}e^{bx^{2}}\,d\rho \left( x\right)  &=&\left(
1-i\frac{d}{dz}\right) \int_{-\infty }^{\infty }e^{izx}e^{- 2 (1-b)x^{2}/2}\,d x  \\
&=& C(\alpha) \left( 1+i\frac{z}{2\alpha }\right) e^{-z^{2}/4\alpha }.
\end{eqnarray*}%
Thus $z=2\alpha i$ is always a non-real zero. This finishes the verification
for this particular example.

\subsubsection%{$\protect\rho $ such that $\protect\int e^{bx^{2}}\,d\protect%
%\rho <\infty $, iff $b\leq b_{0}$.}
{${\cal T}_\rho = (-\infty, b_0]$ with $b_0 \in [0, \infty)$.}

\begin{case}
\label{nexttolastcase}
It is not possible to have ${\cal P}_\rho = (-\infty,b_0]$
or $[\Lambda_0, b_0]$ with $\Lambda_0 \in (-\infty, b_0)$.
%There exists $\Lambda _{0}>-\infty $, such that $\left\{ b:\int
%e^{izx}e^{bx^{2}}\,d\rho \text{ has only real zeros}\right\} =[\Lambda
%_{0},b_{0}]$
\end{case}

We rule out the possibility of such probability distributions by applying the
weak convergence result, Theorem~\ref{weak}. Indeed, suppose $\rho $ is a
probability distribution that satisfies the above conditions. Then for any $\lambda
\in (-\infty, b_0]$ or $ \in \lbrack \Lambda _{0},b_{0}]$, the normalized measure 
\[
d\mu _{\lambda }=\frac{e^{\lambda x^{2}}d\rho \left( x\right) }{%
\int_{-\infty }^{\infty }e^{\lambda x^{2}}d\rho \left( x\right) } 
\]
satisfies all three conditions of Definition \ref{LYp}. Let $\lambda
_{n}=b_{0}-1/n$, so that $\mu _{\lambda _{n}}\in \mathcal{X}$. Applying Theorem %
\ref{weak},  we conclude that as $n\rightarrow \infty $, $\mu _{b_{0}}\in 
\mathcal{X}$. But this would contradict our assumptions 
since Property~2 of Definition~\ref{LYp}
would then imply that 
$\int_{-\infty }^{\infty }e^{bx^{2}}d\rho \left( x\right) <\infty $
for some $b>b_{0}$.

\begin{case}
%$\left\{ b:\int e^{izx}e^{bx^{2}}\,d\rho \text{ has only real zeros}\right\}
%=\left\{ b_{0}\right\} $.
${\cal P}_\rho = \{ b_0 \}$.
\end{case}

We give an example with %$\left\{ b:\int e^{izx}e^{bx^{2}}\,d\rho \text{ has
%only real zeros}\right\} =\left\{ 0\right\} $. 
$b_0 = 0$. Let $Y$ and $Y^{\prime }$ be
i.i.d. Poisson random variables with mean $1/2$, so that $\mathbb{E}\left[
e^{zY}\right] =e^{\frac{1}{2}\left( e^{z}-1\right) }$. Let $W=Y-Y^{\prime }$;
then $W$ is a symmetric random variable with $\mathbb{E}\left[ e^{zW}%
\right] =\exp \left( \cosh z-1\right) $. We denote by $d\mu _{W}$ its
probability distribution.
Take a random variable $X$ with distribution $\frac{1+x^{2}}{2}d\mu _{W}$; then 
\begin{eqnarray*}
\mathbb{E}\left[ e^{zX}\right] &=&\frac{1}{2}\int e^{zx}\left(
1+x^{2}\right) d\mu _{W}\left( x\right) =\frac{1}{2}\left( 1+\frac{d^{2}}{%
dz^{2}}\right) \mathbb{E}\left[ e^{zW}\right] \\
&=&\frac{1}{2}\cosh z\cdot \left( 1+\cosh z\right) \exp \left( \cosh
z-1\right) ,
\end{eqnarray*}%
whose zeroes are $\pm i\pi /2,\pm i3\pi /2,...$ and $\pm i\pi ,\pm i3\pi
,... $ which are all pure imaginary. The fact that for any $b <0$, some
zeros of $H_\rho (z)$ are not purely real follows as in Case~7.

We note that one can obtain an example of Case~8 with $b_{0}>0$ by
replacing the above distribution $d \mu_X$ with 
$\text{Const.} \exp{(-b_0 x^2)} d \mu_X$.
%the random variable $X$ with some Gaussian random variable.

\begin{case}
\label{lastcase}
%$\left\{ b:\int e^{izx}e^{bx^{2}}\,d\rho \text{ has only real zeros}\right\}
${\cal P}_\rho =\emptyset .$
\end{case}

Examples of such probability distributions have been given in the discussion
after Proposition \ref{noPIZ}.

\paragraph{\textbf{Acknowledgments:}}

The research reported here was supported in part by U.S. NSF grant
DMS-1507019. We thank Ivan Corwin for
the invitation to submit this paper to Bulletin of the AMS and for
comments on an earlier draft. We also thank Louis-Pierre Arguin and two anonymous reviewers
for their detailed comments and suggestions which have been incorporated
into the current draft.

\bibliographystyle{alpha}
\bibliography{DBN}

\end{document}